\documentclass[11pt]{amsart}

\usepackage[margin=1.12in]{geometry}
\usepackage{amsmath,amssymb,amsthm}
\newcommand{\probP}{\text{I\kern-0.15em P}}
\usepackage{microtype}
\usepackage[colorlinks=true,linkcolor=black,citecolor=black,urlcolor=black]{hyperref}
\hypersetup{
  pdftitle={Almost Every Graph Is Reconstructible from Its Token Graphs},
  pdfauthor={Yuanming Luo}
}

\newtheorem{theorem}{Theorem}
\newtheorem{lemma}[theorem]{Lemma}

\title{Almost Every Graph Is Reconstructible from Its Token Graphs}
\author{Yuanming Luo}
\date{August 2026}

\begin{document}

\begin{abstract}
Let $F_k(G)$ denote the $k$-token graph of a finite graph $G$. The token graph reconstruction conjecture asks whether, for any $1\le k<|V(G)|$, $F_k(G)\cong F_k(H)$ implies $G\cong H$. This paper proves that the conjecture holds for almost every graph.
\end{abstract}

\maketitle

Let $G$ be a finite graph of vertices set $[n]$ and let $1 \leq k < n$. The $k$-token graph $F_k(G)$ has vertex set $\binom{[n]}{k}$, where two vertices $A \sim B$ if and only if $A \triangle B \in E(G)$. The $k$-token graph was defined and systematically studied by Fabila-Monroy et al.~\cite{FabilaMonroyEtAl2012}. The construction had previously been studied for $k=2$ as the double vertex graph by Alavi, Behzad, Erd\H{o}s, and Lick~\cite{AlaviBehzadErdosLick1991}. Reconstruction from $F_2(G)$ was considered by Jacob, Goddard, and Laskar~\cite{JacobGoddardLaskar2007}; Fabila-Monroy and Trujillo-Negrete proved the conjecture for connected induced-$(C_4,\mathrm{diamond})$-free graphs in a preprint~\cite{FabilaMonroyTrujilloNegrete2022}. The use of common neighbors in this context was introduced in~\cite{JacobGoddardLaskar2007}.

There is a connection between the $2$-token graph reconstruction conjecture and the graph reconstruction conjecture: for every vertex $v$, the card $G-v$ is an induced subgraph of $F_2(G)$~\cite{FabilaMonroyEtAl2012}. Thus, if one could organize the deck of a graph into its $2$-token graph $F_2(G)$, token graph reconstruction might yield graph reconstruction. However, reconstructing the $2$-token graph from the deck is itself difficult.

We call a pair $\{u,v\}$ \emph{admissible} if $|N_G(u) \cap N_G(v)| \geq 5$. Let $A(G)$ be the set of all admissible pairs of $G$. Define the \emph{admissible augmentation} of $G$ by $C(G)=(V(G),E(G)\cup A(G))$. Let $G_n\sim G(n,1/2)$ be the random graph on $[n]$. We will prove that if $C(G_n)=K_n$, then $G_n$ satisfies the token graph reconstruction conjecture. It is therefore enough to show that $\probP(C(G_n)=K_n)\to 1$ as $n\to\infty$. Before proving this, we first establish the following lemma.

\begin{lemma}\label{lem:commutation}
For every graph $G$ and every $1\leq k<|V(G)|$, we have $C(F_k(G))=F_k(C(G))$.
\end{lemma}

\begin{proof}
Let $U,V\in V(F_k(G))$ be distinct. Suppose $U\triangle V=\{u,v\}$. Write $R=U\cap V$, so $U=R\cup\{u\}$ and $V=R\cup\{v\}$. Let $w$ be a common neighbor of $u$ and $v$ in $G$. If $w\notin R$, then $R\cup\{w\}$ is adjacent to both $U$ and $V$ in $F_k(G)$. If $w\in R$, then $(R\setminus\{w\})\cup\{u,v\}$ is adjacent to both $U$ and $V$. Thus every common neighbor of $u$ and $v$ gives a common neighbor of $U$ and $V$.

Conversely, let $W$ be a common neighbor of $U$ and $V$ in $F_k(G)$. Since $W$ differs from both $U$ and $V$ by one token move, either $W=R\cup\{w\}$ for some $w\notin R\cup\{u,v\}$, or $W=(R\setminus\{w\})\cup\{u,v\}$ for some $w\in R$. In the first case, adjacency of $W$ to $U$ and $V$ implies $uw,vw\in E(G)$. The second case gives the same conclusion. Hence $w$ is a common neighbor of $u$ and $v$.

Thus $U$ and $V$ have the same number of common neighbors in $F_k(G)$ as $u$ and $v$ have in $G$. Moreover, $U$ and $V$ are adjacent in $F_k(G)$ if and only if $u$ and $v$ are adjacent in $G$. Therefore $\{U,V\}\in E(C(F_k(G)))$ if and only if $\{u,v\}\in E(C(G))$.

Let $n=|V(G)|$. Now suppose $|U\triangle V|\geq 4$. If $|U\triangle V|=4$, then $U$ and $V$ have at most four common neighbors in the Johnson graph $J(n,k)$, and therefore at most four common neighbors in $F_k(G)$. If $|U\triangle V|\geq 6$, they have no common neighbor in $J(n,k)$, since two token moves can relocate at most two tokens. Thus $C(F_k(G))$ does not add an edge between $U$ and $V$. They are also nonadjacent in $F_k(C(G))$ because their symmetric difference has size at least four. Hence the two graphs have the same edge set.
\end{proof}

Every automorphism of $J(n,k)$ is induced by a permutation of $[n]$, except when $n=2k$, in which case it may also be followed by complementation~\cite{RamrasDonovan2011,Ganesan2018}. When $n=2k$, complementation is an automorphism of every $F_k(G)$ because $U^c\triangle V^c=U\triangle V$. We will use automorphisms of $J(n,k)$ to obtain an isomorphism between $G$ and $H$.

\begin{theorem}\label{thm:deterministic}
If $C(G)=K_n$, then $G$ satisfies the token graph reconstruction conjecture.
\end{theorem}

\begin{proof}
Fix $1\leq k<n$, and suppose $F_k(G)\cong F_k(H)$. Let $h=|V(H)|$. Since $|V(F_k(G))|=\binom nk$ and $|V(F_k(H))|=\binom hk$, we have $h=n$.

Let $\alpha:F_k(G)\to F_k(H)$ be an isomorphism. Since an isomorphism preserves common-neighbor counts, the same bijection is an isomorphism from $C(F_k(G))$ to $C(F_k(H))$. By Lemma~\ref{lem:commutation} and $C(G)=K_n$, this gives $J(n,k)\cong F_k(C(H))$. By the token-edge count~\cite{FabilaMonroyEtAl2012}, $C(H)$ has $\binom n2$ edges, so $C(H)=K_n$. Therefore $\alpha$ is an automorphism of $J(n,k)$.

After composing with complementation when $n=2k$, if necessary, there is a permutation $\pi$ of $[n]$ such that $\alpha(U)=\pi[U]$ for every $U\in\binom{[n]}k$. For distinct $u,v\in[n]$, choose $R\in\binom{[n]\setminus\{u,v\}}{k-1}$ and put $U=R\cup\{u\}$ and $V=R\cup\{v\}$. Then
$$
uv\in E(G)\iff U\sim V\iff \pi[U]\sim\pi[V]\iff \pi(u)\pi(v)\in E(H).
$$
Thus $\pi$ is an isomorphism from $G$ to $H$.
\end{proof}

It remains to show that $C(G_n)=K_n$ with probability tending to one. For fixed distinct $u,v\in[n]$, every other vertex is a common neighbor of $u$ and $v$ with probability $1/4$, independently of the others. Hence $|N_{G_n}(u)\cap N_{G_n}(v)|\sim\operatorname{Bin}(n-2,1/4)$, and
\[
\probP\bigl(|N_{G_n}(u)\cap N_{G_n}(v)|\leq 4\bigr)
=\left(\frac34\right)^{n-2}\sum_{j=0}^4\binom{n-2}{j}3^{-j}
=O\!\left(n^4\left(\frac34\right)^n\right).
\]
If $C(G_n)\neq K_n$, then some pair has at most four common neighbors. A union bound over the $\binom n2$ pairs gives
\[
\probP(C(G_n)\neq K_n)=O\!\left(n^6\left(\frac34\right)^n\right)\to 0.
\]
Therefore $\probP(C(G_n)=K_n)\to 1$, and Theorem~\ref{thm:deterministic} shows that the token graph reconstruction conjecture holds for almost every graph.

\end{document}